\documentclass[12pt,reqno]{amsart}  

\usepackage{ifthen}
\newboolean{plots_included}
\setboolean{plots_included}{true}
\usepackage{float}

\usepackage{amssymb,amsbsy,amsfonts,amsthm,latexsym,amsopn,amstext,amsxtra,euscript,amscd,amsthm,url}
\usepackage{epsfig, subfigure, afterpage}
\usepackage[dvipsnames, x11names]{xcolor}
\usepackage{hyperref}
\usepackage{enumitem}

\usepackage[vmargin=2.2cm,hmargin=2.6cm]{geometry} 

\newtheoremstyle{plaintheorems}% name
{10pt}%  Space above
{6pt}%  Space below
{}% Body font
{}% Indent amount (empty = no indent, \parindent = para indent)
{\bfseries}% Thm head font
{.}%Punctuation after thm head
{.5em}% Space after thm head: " " = normal interword space;
{\thmname{#1}\thmnumber{ #2}\thmnote{ (#3)}}
\theoremstyle{plaintheorems}
\newtheorem{Rem}{Remark}

\newtheoremstyle{sltheorems}% name
{10pt}%  Space above
{6pt}%  Space below
{\slshape}% Body font
{}% Indent amount (empty = no indent, \parindent = para indent)
{\bfseries}% Thm head font
{.}%Punctuation after thm head
{.5em}% Space after thm head: " " = normal interword space;
{\thmname{#1}\thmnumber{ #2}\thmnote{ (#3)}}
\theoremstyle{sltheorems} 
\newtheorem{Thm}{Theorem}
\newtheorem{CThm}{Classical Theorem}
\newtheorem*{Thm*}{Theorem}

\newtheorem{lem}{Lemma}

\newtheorem*{cor*}{Corollary}

\newcommand{\C}{\mathbb{C}}

\newcommand{\eps}{\ensuremath{\varepsilon}}

\newcommand*\kronecker[2]{%
\relax\if@display
\expandafter{(\frac{#1}{#2})}
\else
\expandafter{(#1/#2)}
\fi
}
\makeatother

\makeatletter
\let\@@pmod\pmod
\DeclareRobustCommand{\pmod}{\@ifstar\@pmods\@@pmod}
\def\@pmods#1{\mkern4mu({\operator@font mod}\mkern 6mu#1)}
\makeatother

\newcommand\pieter[1]{{\color{blue}{#1}}}

\newcommand\ale[1]{{\color{DeepSkyBlue4}{#1}}}

\allowdisplaybreaks

\usepackage{mathtools} 

\newcommand{\Euler}{{\pmb \gamma}}

\newcommand{\BB}{\mathcal{B}}
\newcommand{\CC}{\mathcal{C}}
\newcommand{\MM}{\mathcal{M}}
\newcommand{\RR}{\mathcal{R}}

\newcommand{\Reg}{\operatorname{Reg}}

\usepackage{microtype} %  (ale: insert Subliminal refinements towards typographical perfection) 
\usepackage{geometry}
\usepackage{caption} 
\makeatletter
\patchcmd{\env@cases}{1.2}{1}{}{}
\makeatother

\begin{document}

\title[Brauer-Siegel ratio for prime cyclotomic fields]
{The Brauer-Siegel ratio for prime cyclotomic fields}

\author[N. Kandhil, A. Languasco and P. Moree]{Neelam Kandhil, Alessandro Languasco and Pieter Moree}

\subjclass[2010]{Primary 11R18, 11R29; Secondary 11R47, 11Y60}
\keywords{cyclotomic fields, residues, class number, Siegel zero, Dirichlet $L$-series}

\begin{abstract}
The Brauer-Siegel theorem concerns the size
of the product of the class number and the regulator of a number
field $K$. We derive bounds for this product in case $K$ is
a prime cyclotomic field, distinguishing between
whether there is a Siegel zero or not.
In particular, we make a result of Tatuzawa (1953) more explicit. 
Our theoretical advancements are complemented by numerical illustrations 
that are consistent with our findings. 
\end{abstract}
  
\maketitle  

\section{Introduction}

Let $K$ be a number field, ${\mathcal O}$ its ring of integers and $s$ a complex variable.
For $\Re(s)>1$ the 
\emph{Dedekind zeta function} is defined by
\begin{equation*}
%\label{KEuler}
\zeta_K(s)=\sum_{\mathfrak{a}} \frac{1}{N{\mathfrak{a}}^{s}}
=\prod_{\mathfrak{p}}\frac{1}{1-N{\mathfrak{p}}^{-s}},
\end{equation*}
where $\mathfrak{a}$ ranges over 
the non-zero ideals in ${\mathcal O}$, $\mathfrak{p}$ 
over the prime ideals in ${\mathcal O}$, and $N{\mathfrak{a}}$ denotes the 
\emph{absolute norm}
of $\mathfrak{a}$,
that is the index of $\mathfrak{a}$ in $\mathcal O$.
Note that $\zeta_{\mathbb Q}(s)$ is merely the Riemann zeta-function $\zeta(s)$.
It is known that $\zeta_K(s)$ can be analytically continued to $\C \setminus \{1\}$,
and that it has a simple pole at $s=1$. It has residue 
\begin{equation}
    \label{Kresidue}
\RR(K)=\frac{2^{r_1}(2\pi)^{r_2} h(K)\Reg(K)}{\omega_K \sqrt{d_K}},
\end{equation}
where $r_1$ and $r_2$ denote the number of real, 
respectively complex embeddings of $K$, $d_K$ the absolute value of 
the discriminant, $\omega_K$ the roots of unity 
in $K$, $\Reg(K)$ its regulator and $h(K)$ its class number.
Formula \eqref{Kresidue} is called the \emph{analytic class number formula}.
In it the only mysterious quantity
is $h(K)\Reg(K)$ and one could hope to get bounds on it via estimates of $\RR(K)$.
One has for example, under the Generalized Riemann Hypothesis 
and the strong Artin conjecture for \(\zeta_K(s)/\zeta(s)\), 
\begin{equation}
\label{conditional-estims}
\Bigl(\frac{1}{2} + o(1)\Bigr) \frac{\zeta(n)}{e^\Euler \log\log d_K} 
\le 
\RR(K) \le (2 + o(1))^{n-1} \bigl(e^\Euler \log\log d_K\bigr)^{n-1},
\end{equation}
where $n$ denotes the degree of $K$, see \cite[Section~3]{chokim},  and $\Euler$ is  Euler's constant. 
In 2015, Louboutin \cite{Lou15} (see also \cite{Li12}) gave a weaker, but
unconditional, bound for $\RR(K)$. 
More precisely, he demonstrated that for any given $\eps > 0$ 
and $n_0 \geq 5$, there exists a number $\rho_0$ such that for 
all number fields $K$ of degrees $n \geq n_0$ and 
$d_K \geq \rho^{n}_0$, we have
\begin{equation}\label{messybound}
    \mathcal{R}(K) \leq \frac{3}{\sqrt{n}}    \Bigl( \frac{c(1+\eps) e^{\Euler +\sqrt{6/n}} \log d_K}{n}\Bigr)^{n-1},
\end{equation}
where $c= \frac{1}{2} \bigl(1- \frac{1}{\sqrt{5}}\bigr)$.

From now on our focus is exclusively on cyclotomic fields $K= {\mathbb Q}(\zeta_q)$ of
prime conductor $q \ge 3$. Then $d_K = q^{q-2}$ and writing $h(q)$ for $h(K)$ and
$\Reg(q)$ for $\Reg(K)$,  
we have
\begin{equation}
\label{oud}
\zeta_K (s)=\zeta(s)\prod_{\chi\ne \chi_0}L(s,\chi),
\end{equation}
where $\chi$ runs over the non-principal characters modulo $q$
and $L(s,\chi)$ denotes a Dirichlet $L$-function. 
This identity in combination with \eqref{Kresidue} gives
\[
\RR(q) = \prod_{ \chi \ne \chi_0} L(1,\chi) = \frac{h(q)\Reg(q)}{H(q)},
\]
where
\(
%\label{Hq-def}
H(q) = 2 \sqrt{q} \bigl(\frac{q}{2\pi}\bigr)^{\frac{q-1}{2}}.
\)

In analogy with the terminology used for the relative class number, we will 
call $\RR(q)$ the \emph{Brauer-Siegel ratio}
for a prime cyclotomic field, a term that seems to have been introduced by Ulmer \cite{Ulmer}
in the context of abelian varieties over function fields.
In his Theorem 3, Tatuzawa \cite{Tatu} proved that
for every positive $\eps$ there exists $c(\eps)>0$ such
that 
\begin{equation}
\label{Tatu-thm3}
\frac{c(\eps)}{q^\eps} < \RR(q) < (\log q)^c,
\end{equation}
where $c>0$ is an absolute constant.

Here, adapting the technique used in Kandhil et al.\,\cite{paper1-rq} to study
the order of magnitude of the Kummer ratio\footnote{This analogy also 
speaks in favor of the terminology Brauer-Siegel ratio.} for the relative class
number of prime cyclotomic field, we 
bound the value of $c$ in \eqref{Tatu-thm3}.
We show that $c$ is essentially 
at most $2$ in the most general case; 
otherwise it is less than $1$.
Moreover, we explicitly show the role
of the \emph{Siegel zero} 
in both the bounds appearing in \eqref{Tatu-thm3}.
Stark \cite[Lemma~3]{stark}\footnote{In fact Stark's region
is not the largest known, see, e.g., Louboutin \cite{Lou15b}.}, showed that 
$\zeta_K(s)$ (with $K\ne \mathbb Q$) has at most one zero in the region in the complex plane determined by
$$
\Re(s) 
\ge
1 - \frac{1}{4 \log d_K},
\quad \quad
|\Im(s)| \le \frac{1}{4 \log d_K}.
$$

If such a zero exists, it is
real, simple and often called 
\emph{Siegel zero}. 
When $K = \mathbb{Q}(\zeta_q)$, by  \eqref{oud},  such a Siegel zero is  a zero
of a Dirichlet $L$-series attached to a real and quadratic 
\emph{exceptional character}\!\!~$\pmod*{q}$.

In our result the Siegel zero contribution will be expressed using the \emph{exponential
integral function} 
\begin{equation}
\label{E1-series}
E_1(x) :=
\int_{x}^{\infty} e^{-t}\frac{dt}{t}
=
- \Euler -\log x  +\int_0^x (1-e^{-t})\frac{dt}{t}= 
- \Euler - \log x
- \sum_{k=1}^{\infty} \frac{(-x)^k}{(k!)k}\quad (x>0).
\end{equation}

We are now ready to state our main theorem that makes
Tatuzawa's one \cite[Theorem~3]{Tatu} more explicit.

\begin{Thm}
\label{rq-direct3}
Let $\ell(q)$ be a function that tends arbitrarily slow and monotonically 
to infinity as $q$ tends to infinity.
There is an effectively computable prime
$q_0$ (possibly depending on $\ell$) 
such that the following statements are true:
\begin{enumerate}[label={\arabic*)}, ref=\arabic*, wide, nosep, itemindent= 4pt, after=\vspace{-8pt}]
\item
\label{nosiegelzero3}
If for some $q\ge q_0$ the family 
of Dirichlet $L$-series $L(s, \chi)$, with $\chi$ any non-principal character modulo $q$, 
has no Siegel zero, then 
\begin{equation}
\label{nosiegelzero-estims}
 \frac{e^{-1.87}}{(\log q)^{1-\xi}}
 <
\prod_{ \chi \ne \chi_0} L(1,\chi) 
 < 
e^{0.51}
\, (\log q)^{1-\xi},
\end{equation} 
for some absolute constant $\xi$.
\item
\label{Siegelzero3} 
If for some $q\ge q_0$ the family 
of Dirichlet $L$-series $L(s, \chi)$, with $\chi$ any non-principal character modulo 
$q$, has a Siegel zero $\beta_0$ then
$$
\frac{e^{-1.87}e^{-E_1(1- \beta_0)} }{
\, (\log q)^2 \, 
\ell(q)}
<
\prod_{ \chi \ne \chi_0} L(1,\chi)
<  
e^{0.51}e^{-E_1(1- \beta_0)} 
\, (\log q)^2 \, 
\ell(q).
$$ 

\end{enumerate}
\end{Thm}

\newcommand{\bound}{10^7}
We were able to perform extensive computations for the odd 
primes up to $10^7$ using the Fast Fourier Transform method
already presented in \cite{Languasco2021a}, \cite{Languasco2021}, see also \cite{LanguascoR2021}. They
show a remarkable fit between $\RR(q)=\prod_{ \chi \ne \chi_0} L(1,\chi)$ and 
$c/(\log q)^{3/4}$, with $c\in(1/5, 2/3)$, see Figure \ref{fig1}. In this respect,
the scatter plot of the normalized values $\RR(q) (\log q)^{3/4}$
presented in Figure \ref{fig2} is particularly relevant.
We think it is possible that the ``true'' order of magnitude for $\RR(q)$ 
in Theorem \ref{rq-direct3} might be the one on the left hand side of \eqref{nosiegelzero-estims} 
with $\xi=1/4$.
In Figure \ref{fig3} we show the histograms obtained using the values presented
into the first two figures.

\begin{Rem}
%\label{rmk41}
All constants in  Theorem \ref{rq-direct3} can be further sharpened 
by arguing as in Remark \ref{improved-constants} below.
\end{Rem}

\begin{Rem}
It is a consequence of Theorem \ref{rq-direct3} that asymptotically 
the upper bounds \eqref{conditional-estims} and \eqref{messybound} 
are quite weak for prime cyclotomic fields.
However, the \emph{lower} bound in \eqref{conditional-estims} seems reasonable sharp
in this case.
\end{Rem}

\begin{Rem}
\label{E1bound}
We have
$1 \ll E_1(1- \beta_0) < \eps \log q + c(\eps)$,
where  $c(\eps)$ is ineffective.
Since 
$0<1-e^{-x}< x$, it follows from the first equality in \eqref{E1-series} that
$$-\Euler - \log x < E_1(x) < -\Euler - \log x + x\quad (x>0).$$
On
using that for every $\eps>0$ there exists a constant $c_1(\eps)$ such that
$\beta_0 < 1- c_1(\eps)q^{-\eps}$,
the bounds for $E_1(x)$ lead to
$1 \ll E_1(1- \beta_0) < \eps \log q + c(\eps),$
where $c(\eps)$ is ineffective. 
Using the weaker,
but with an effective constant, estimate $\beta_0 < 1- cq^{-1/2}(\log q)^{-2}$ we obtain that 
$1 \ll E_1(1- \beta_0) < \frac{1}{2} \log q + 2\log \log  q + c_1$,
where $c_1>0$ is an effective constant.
We also recall that Bessassi \cite[Theorem~17]{Bessassi} proved that 
$\beta_0 < 1- 6/(\pi \sqrt{q})$ for $q\equiv 3 \pmod*{4}$
and hence in this case one obtains
$1 \ll E_1(1- \beta_0) < \frac{1}{2} \log q + \log (\pi/6)$.
\end{Rem}

Clearly Theorem \ref{rq-direct3} has implications for the 
asymptotic estimates of $h(q)\Reg(q)$; for example, Part \ref{Siegelzero3} 
and the estimates of Remark \ref{E1bound} yield  
\[
\log(h(q) \Reg(q)) 
=
\frac{q}{2} \log q- \frac{q}{2} \log (2\pi) + O(\log \log q)
\quad (q\rightarrow \infty),
\]
improving on the Brauer-Siegel implication 

\begin{equation*}
%\label{BSprimecyclotomic}    
\log(h(q) \Reg(q)) \sim \log \sqrt{d_q} \sim \frac{q}{2}\log q \quad (q\rightarrow \infty).
\end{equation*}

The paper is organized as follows:
In Section \ref{sec:prelim} 
we recall results we need (mainly from prime number theory) 
and in Section \ref{useful_lemma} we prove a useful lemma about a sum
over prime powers in an arithmetic progression modulo a prime $q\ge 3$.
Section \ref{Thm1-proof} is devoted to the proof of 
Theorem \ref{rq-direct3}.
Section \ref{sec:numerical} describes  an efficient algorithm to compute $\RR(q)$
and some graphical representations
regarding its distribution.
Section \ref{sec:mertens-analogies} establishes 
some analogies between the prime sums over characters connected 
with $\log \RR(q)$ and the ones for the Mertens' constants in arithmetic
progressions.

\section{Preliminaries}
\label{sec:prelim}
\subsection{Notations}
Throughout this article, we will use the following standard notations:
\begin{equation*}
\pi(t) = \sum_{p \le t} 1, \quad \quad\pi(t;d,b) 
= 
\sum_{\substack{p \le t \\ p \equiv b  \pmod*{d}}} 1,
\end{equation*}
\begin{equation*}
\theta(t;d,b) = \sum_{\substack{p \le t \\ p \equiv b  \pmod*{d}}} \log p \quad  \text{and} \quad
\psi(t;d,b) = \sum_{\substack{n \le t \\ n \equiv b  \pmod*{d}}} \Lambda(n),   
\end{equation*}
where $\Lambda$ denotes the von Mangoldt function 
and $b$ and $d$ are coprime.

\subsection{Siegel zeros}
%\label{sec:Siegel-zero}
\iffalse
Let $K \ne \mathbb{Q}$ be an algebraic number field having $d_K$ as its 
absolute discriminant over the rational numbers. Then, see
Stark \cite[Lemma~3]{stark}, $\zeta_K(s)$ has at most one zero i
n the region in the complex plane determined by
$$
\Re(s) 
\ge
1 - \frac{1}{4 \log d_K},
\quad \quad
|\Im(s)| \le \frac{1}{4 \log d_K}.
$$

If such a zero exists, it is
real, simple and often called 
\emph{Siegel zero}. 
When $K = \mathbb{Q}(\zeta_q)$,
using \eqref{oud} it is easy to see that the Siegel zero is attached to the family of
Dirichlet $L$-series $\pmod*{q}$. In this case, the Dirichlet character $\chi$
such that $L(s,\chi)$ has the Siegel zero is called the \emph{exceptional character}
and it is a well known fact that it is quadratic. 
\fi

The presence of a Siegel zero strongly influences the 
distribution of the primes in the progressions modulo $q$. 
We present two classical results in this direction we will make use of.
\begin{CThm}[Brun-Titchmarsh\footnote{For a proof, see, e.g., 
Montgomery-Vaughan \cite[Theorem~2]{MVsieve}.}]
%\label{BT-thm}
Let $x,y>0$ and $a,q$ be positive integers such that $(a,q)=1$.
Then
\begin{equation}
\label{BT-estim}
\pi(x+y;q,a) - \pi(x;q,a) < \frac{2y}{\varphi(q) \log(y/q)},
\end{equation}
for all $y >q$.
\end{CThm}

In particular, a key role is played by the constant $2$ present in \eqref{BT-estim}; 
from the works of Motohashi \cite{Motohashi1979}, Friedlander-Iwaniec
\cite{FriedlanderI1997}, Ramar\'e \cite[Theorems~6.5-6.6]{Ramare2009}  and Maynard
\cite{Maynard2013}, it is well known that replacing such a constant with any value less
than $2$ is equivalent with assuming that there does not exist a Siegel zero 
for $\prod_{\chi\ne \chi_0}L(s,\chi)$.

In Part \ref{nosiegelzero3} of Theorem \ref{rq-direct3}
we will in fact assume that $\prod_{\chi\ne \chi_0}L(s,\chi)$ has no Siegel zero
and we will make use of the following result 
by Maynard \cite[Proposition~3.5, second part]{Maynard2013}.

\begin{Thm*}[Maynard]
%\label{Maynard-prop}
There is a fixed constant $\eps>0$ such that
there exists an effectively computable constant $q_1$, such that  if
the set of the non-principal Dirichlet $L$-functions
$\pmod*{q}$, for $q>q_1$, does not have a Siegel zero then
for $x \ge q^{7.999}$ and for any $b$ co-prime with $q$ we have
that
\begin{equation}
\label{BT-estim-strong-M-I}
\Big\vert \psi(x;q,b) - \frac{x}{\varphi(q)} \Big\vert 
<  
\frac{(1-\eps)x}{\varphi(q)}.
\end{equation}
\end{Thm*}

From now on $\log_2 x$ denotes $\log\log x$. The following 
theorem of Dusart \cite[Theorem 5.5]{dusart}. will play a crucial role in 
the proof of our main result.
\begin{Thm*}[Dusart]
%\label{dusart-thm}
For $x \geq 2278383$ we have
\begin{equation} \label{newneq}
\Big \vert \sum_{p \leq x} \frac{1}p - \log_2 x - {\mathcal M} \Big \vert \leq \frac{0.2}{(\log x)^3},
\end{equation}
where ${\mathcal M}$, the Meissel-Mertens constant, is 
given by the infinite sum
\begin{equation}
\label{Meissel-Mertens-def}
{\mathcal M}:= \Euler + \sum_{p} \Bigl( \log\Bigl(1-\frac{1}p\Bigr) + \frac{1}{p}\Bigr)
=\Euler - \sum_{p\ge 2}\sum_{m \ge 2}  \frac{1}{m p^m} .
\end{equation}
\end{Thm*}
One has
${\mathcal M}\approx 0.26149 72128 47643$, for more decimals see
\url{https://oeis.org/A077761}.

\section{A useful lemma}
\label{useful_lemma}
For $q$ a prime and $b$ an integer, let
\begin{equation}
\label{Sa-def}
S_q(b):= \sum_{\substack{m\ge 2 \\ p^m\equiv b \pmod*{q}}}\!\! \frac{1}{m p^m},
\end{equation}
where the sum is over all pure prime powers that are congruent to $b \pmod*{q}$.
This quantity will play a role in the proof of Theorem \ref{rq-direct3}.

\iffalse
\pieter{Note that 
$$S_q(b)=\sum_{n\equiv b \pmod*{q}}\frac{\Lambda(n)}{n\log n}.$$}
\fi

We will need the following lemma, 
the proof of which is similar to a well known 
result by Ankeny and Chowla, see the estimate of $C_4$ in \cite{AC}.
\begin{lem}\label{lem1-referee}
Put $\alpha(m):= \frac{1}{2}(m^2 - m), \beta(m):=\frac{1}{2}(m^2 + m) - 1$
and 
\begin{equation}
\label{A-def}
{\mathcal A}  := 
\sum_{m\ge 2} \frac{1}{m} \!\sum_{k = \alpha(m)}^{\beta(m)} 
\frac{1}{k}.
%= 1.600088\dotsc 
\end{equation}
For any odd prime number $q$ and for every $b$ coprime to $q$, we have
\begin{equation}
\label{Rq-def}
R(q,b)
:= (q-1) S_q(b)
%\sum_{m \ge 2} \sum_{p^m\equiv \ale{b} \pmod*{q}} \frac{1}{m p^m}
\le
{\mathcal A}  +\Bigl(\frac{\pi^2}{6} -{\mathcal A} \Bigr) \frac{1}{q},
\end{equation}
where $S_q(b)$ is defined in \eqref{Sa-def}.
In particular, $R(q,b) \le 1.608$ 
for $q\ge 7$ and every $b$ coprime to $q$.
\end{lem}

\begin{proof}
Note that without loss of generalization we may assume that $1\le b\le q-1$.
The contribution of the terms to $S_q(b)$ with $2\le p \le q+1$ and $p \ge q+2$, 
we denote by $H$, respectively $T$.
We now proceed to bound the tail $T$.
For a given $m\ge 2$, let $x_{m,j}$, $1\le j \le f(m)$,
denote the $f(m)$ integral solutions in $\{2,\dotsc, q+1\}$
of $x^m\equiv b \pmod*{q}$. 
Note that $f(m)\le m$.
Since $p$ must be equal to 
one of the $x_{m,j} \pmod*{q}$, $p\ge q+2$ is greater
than any $x_{m,j}$, we have
\begin{equation}
\label{tail-estim}
T 
\le
\sum_{m \ge 2} 
\sum_{j=1}^{f(m)}
\sum_{k\ge 1}
\frac{q-1}{m (x_{m,j}+kq)^m}
\le
\sum_{m \ge 2} 
\sum_{k\ge 1}
\frac{q-1}{(kq)^m}
\le
\zeta(2)
\sum_{m \ge 2} 
\frac{q-1}{q^m}
=
\frac{\zeta(2)}{q}.
\end{equation}

We now bound the head $H$.
For a given $m\ge 2$, let $g(m)$ denote the number of solutions in primes contained in
$\{2,\dotsc, q+1\}$ of $x^m\equiv b \pmod*{q}$.
Clearly $g(m)\le f(m)\le m$.
Due to the weight $1/m$ in the definition of $S_q(b)$, it is more unfavorable to have a 
square, say, followed by a cube, than the other way around in the
progression $b+q,b+2q,\ldots$. Thus we may assume we have $g(2)$ squares, 
followed by $g(3)$ cubes and so on.
Again due to the weight $1/m$, 
the most unfavorable situation arises if $g(m)=m$, that is two squares followed by three cubes and so on.
Since $b+kq\le kq$, we then find
\begin{equation*}
H\le \frac{q-1}{q}
\sum_{m\ge 2}\frac{1}{m} 
\sum_{k = \alpha(m)}^{\beta(m)} \frac{1}{k}
=\frac{q-1}{q}
{\mathcal A} . 
\end{equation*}

\iffalse
Clearly, $(m^\prime, j^\prime)\ne(m,j)$
implies  $k_{m^\prime, j^\prime}\ne  k_{m,j}$.
Now,
\[
H = \sum_{m\ge 2}\sum_{j=1}^{g(m)}\frac{q-1}{m(\ale{b}+q k_{m,j})}.
\]
Since $m<m^\prime$ and $k<k^\prime$ imply
\[
\frac{q-1}{m(\ale{b}+kq)} + \frac{q-1}{m^\prime(\ale{b}+k^\prime q)}
>
\frac{q-1}{m(\ale{b}+k^\prime q)} + \frac{q-1}{m^\prime(\ale{b}+k q)},
\]
\fi
\iffalse
Thus, in order to bound $H$ from above it suffices to bound it assuming that 
the $k_{m,j}$'s are ordered in lexicographical  order:
$1 \le k_{2,1} < \dotsm <   k_{2,g(2)}
< k_{3,1} < \dotsm <   k_{3,g(3)} < \dotsm $,
and then assuming that the $k_{m,j}$'s are consecutive integers.
Hence, setting
$s_n:=\sum_{m=2}^{n} g(m) \le s_n^\prime :=\sum_{m=2}^{n} m = n(n+1)/2 -1$
for $n\ge 2$ and $s_1=0$, 
%and using $1+kq \ge (q-1)k$ for $k\ge 1$,
and using  \ale{$b+kq\ge kq$} 
we have
\begin{align*}
H 
&\le
\sum_{m\ge 2}\frac{1}{m} \sum_{s_{m-1}<k \le s_{m}} \frac{q-1}{\ale{b}+kq}
\le
\ale{\frac{q-1}{q}}
\sum_{m\ge 2}\frac{S_m-S_{m-1}}{m} 
=
\ale{\frac{q-1}{q}}
\sum_{m\ge 2}\frac{S_m}{m(m+1)} 
\\&
\le
\ale{\frac{q-1}{q}}
\sum_{m\ge 2}\frac{S^{\prime}_{m}}{m(m+1)} 
=
\ale{\frac{q-1}{q}}
\sum_{m\ge 2}\frac{S^{\prime}_{m}-S^{\prime}_{m+1}}{m} 
=
\ale{\frac{q-1}{q}}
\sum_{m\ge 2}\frac{1}{m} 
\sum_{k = \alpha(m)}^{\beta(m)} \frac{1}{k}
=
\ale{\frac{q-1}{q}}
A,
\end{align*}
where
\[
S_{1} :=0, S^{\prime}_{1} :=0,
\ \textrm{and}\ 
S_{m} := \sum_{1\le k \le s_{m}} \frac{1}{k}
\le
\sum_{1\le k \le s^{\prime}_{m}} \frac{1}{k}
= :
S^{\prime}_{m}
\ \textrm{for}\  m \ge 2.
\]
\fi
The result follows on adding $H$ and $T$ and doing some simple numerics.
\end{proof}

\begin{Rem}[Precise numerical approximation of ${\mathcal A}$]
Since $\beta(m) = \alpha(m+1)-1$, we obtain
\begin{equation*}
{\mathcal A} 
=
\sum_{m\ge 2} \frac{1}{m} \bigl(H_{\alpha(m+1)-1} - H_{\alpha(m)-1}\bigr)
=
\sum_{m\ge 2} \frac{H_{\alpha(m+1)-1}}{m^2 + m}
=
\frac{\Euler}{2}
+
\frac{1}{2} \sum_{j\ge 3} \frac{\psi(\alpha(j) )}{\alpha(j)},
\end{equation*}
where $\psi(x)$ is the \emph{digamma function}, 
 $H_n$ denotes the $n$-th harmonic number, $H_0=0$ and we also used
 that $\psi(n)=H_{n-1}-\Euler$ for every $n\ge 1$.
Recalling that $\psi(x) < \log x$, the third formula for ${\mathcal A}$ shows that the series converges,
although not very quickly.
However, it can be used to evaluate ${\mathcal A} $,
since  there exist very fast and accurate algorithms to compute $\psi(x)$ for positive $x$.
For example, truncating the final sum in the expression for ${\mathcal A}$ at $10^{10}$ gives
\begin{equation}
\label{A-eval}
{\mathcal A} \approx 1.6000883438\dotsc
\end{equation}
\end{Rem}

\iffalse
\ale{ALTERNATIVE using $\psi$:
The estimate for $T$ can be improved
by using the fact that
\begin{align*}
\sum_{m \ge 2} 
\sum_{k\ge 1}
\frac{1}{(kq)^m}
&= 
\sum_{m \ge 2} 
\frac{\zeta(m)}{q^m}
=
-\frac{1}{q}\Bigl(\Euler + \psi(1-\frac{1}{q})\Bigr),
\end{align*}
where $\psi$ is the digamma function 
(the Taylor series of $\psi(1-x)$ is $ -\Euler - \sum_{k=1}^{\infty} \zeta(k+1) x^k$),
so that 
\[
T \le 
-\frac{q-1}{q}\Bigl(\Euler + \psi(1-\frac{1}{q})\Bigr).
\]
}
\fi

\begin{Rem}
\label{improved-constants}
The first estimate in \eqref{tail-estim} 
together with $f(m)\le m$ and $x_{m,j}\ge 2$ leads to
\begin{align*}
T  
&\le
(q-1)
\sum_{m \ge 2} 
\sum_{k\ge 1}
\frac{1}{(2+kq)^m}
%=
%(q-1)
%\sum_{m \ge 2} 
%\frac{1}{q^m}
%\Bigl(
%\zeta(m,2/q) - \frac{q^m}{2^m}
%\Bigr)
=
(q-1)
\Bigl(
\sum_{m \ge 2} 
\frac{\zeta(m,2/q)}{q^m}
-
\frac{1}{2}
\Bigr)
\\&=
\frac{q-1}{q}
 \bigl(\psi(\frac{2}{q}) - \psi(\frac{1}{q})\bigr)
-
\frac{q-1}{2},
\end{align*}
where $\zeta(s,x)$ denotes the Hurwitz zeta-function
and we used  \cite[eq.~(4.1)-(4.3)]{Boyadzhiev2007}
to obtain a closed formula for the series involving
the \emph{Hurwitz zeta-function} values.
Inserting this into the body of Lemma \ref{lem1-referee}
we can replace \eqref{Rq-def} with
the sharper (but less elegant) estimate
\[
R(q,b)
\le
\frac{q-1}{q}
\Bigl({\mathcal A} + \psi(\frac{2}{q}) - \psi(\frac{1}{q})\Bigr) -\frac{q-1}{2},
\]
from which one can infer that
$R(q,b)<1.600177$ for every prime $q\ge 3$ and $b$ coprime to $q$,
with the maximum being attained at $q=229$.
\end{Rem}

\section{Proof of Theorem \ref{rq-direct3}}
\label{Thm1-proof}
 
Using the Euler product for $L(1,\chi)$ with $\chi\ne \chi_0$,
and Taylor's formula for $\log(1-u)$, we obtain
\begin{equation}
\label{starting3}
\log \RR (q)
=
- \sum_{\chi \ne \chi_0} \sum_p \log\Bigl(1-\frac{\chi(p)}{p}\Bigr)
=
\sum_{\chi \ne \chi_0} \sum_p \sum_{m\ge1} \frac{\chi(p^m)}{mp^m}
= 
\Sigma_1 + \Sigma_2,
\end{equation}
say, where $\Sigma_1$ is the contribution of the primes ($m=1$)
and $\Sigma_2$ that of the prime powers ($m\ge 2$). 

We first estimate $\Sigma_2$.
Suppose that $(a,q)=(b,q)=1$ and $b\equiv a \pmod*{q}$.
Then, using 
\begin{equation}
\label{orthoall}
  \frac{1}{q-1} \sum_{\chi ~ \text{mod} ~q} \chi(a) = 
\begin{cases}
 1, & a \equiv 1  \pmod*{q},\\
0, &\text{otherwise},
\end{cases}
\end{equation}
we obtain
$$
 \Sigma_2 = 
(q-1)\sum_{\substack{m\ge 2 \\ p^m\equiv 1 \pmod*{q}}}\!\! \frac{1}{m p^m}-
\sum_{\substack{m\ge 2 \\ p \ne q}} \frac{1}{mp^m}.
$$
Recalling \eqref{Sa-def}, it is easy to see that 
\begin{equation*}
%\label{eqb}
\sum_{\substack{m\ge 2 \\ p \ne q}} \frac{1}{mp^m} =    
\sum_{b=1}^{q-1} S_q(b),
\end{equation*}
and hence 
$$
-\frac{1}{q-1}\sum_{b=1}^{q-1} R(q,b) 
= 
-\sum_{b=1}^{q-1} S_q(b) 
< 
\Sigma_2 
< 
(q-1) S_q(1) = R(q,1).
$$
On invoking Lemma \ref{lem1-referee} we then obtain 
\begin{equation}
\label{Sigma2-estim3}
\vert \Sigma_2 \vert
< 
{\mathcal A}  + \frac{\zeta(2)-{\mathcal A}}{q} ,
\end{equation}
where ${\mathcal A}$ is defined in \eqref{A-def} and evaluated in \eqref{A-eval}.

We now proceed to define some quantities that will be useful later to estimate $\Sigma_1$.
For any $b\in\{1,\dotsc,q-1\}$ and $x>0$ let
\begin{equation}
\label{Sbx-def}
S_q(b,x):=\sum_{\substack{p\le x\\ p\equiv b \pmod*{q}}}\!\! \frac{1}{p}
\quad
\textrm{and}
\quad
S(x):= \sum_{\substack{p\le x\\  p \ne q} }\frac{1}{p}.
\end{equation}
Using \eqref{orthoall} again,
for any $x>0$ we have 
\begin{align}
\label{boundf5} 
\sum_{\chi \ne \chi_0} \sum_{p \le x}  \frac{\chi(p)}{p} 
=  
(q-1) 
S_q(1,x) 
- S(x).
\end{align}
As a consequence we obtain
$$
\Sigma_1 =  
\lim_{x\to\infty}
\bigl((q-1)S_q(1,x) - S(x) \bigr).
$$

We begin by estimating  $S(x)$, followed by estimating $S_q(1,x)$ 
(which will bring the possible Siegel zero into play).
{}From now on  we will assume that $q$ is a sufficiently large prime.
Substituting $x = x_1 = q^{\ell(q)}$ in \eqref{newneq}, we obtain
\begin{equation}\label{112}
S(x_1)
\ge
\log_2 q + \log \ell(q)  + 0.261497
+ \frac{1}{\ell(q) \log q}
- \frac{1}{q}
\end{equation}
and
\begin{equation}\label{1123}
S(x_1)
%\sum_{\substack{p \le x_1 \\ p \ne q}} \frac{1}{p}}
\le
%\sum_{p \le x_1} \frac{1}{p} 
%\le 
\log_2 q + \log \ell(q) + 0.261498 + \frac{1}{ \log q}.
\end{equation}

We will use \eqref{Sigma2-estim3} and \eqref{112}-\eqref{1123}  in the 
proofs of both parts of Theorem \ref{rq-direct3}.

\subsection{Proof of Theorem \ref{rq-direct3}, Part  \ref{Siegelzero3}}
\label{proof_part2}
We first prove Part  \ref{Siegelzero3}. The starting point is \eqref{starting3}.
We split the prime sum $\Sigma_1$ in three 
subsums $S_1, S_2, S_3$ defined according
to whether $p\le x_1$, $x_1 < p \le x_2$ or $p\ge x_2$, 
with $x_2= e^q$ and $x_1= q^{\ell(q)}$.

We start by estimating $S_1$:
{}recalling \cite[eq.~(26)-(27)]{paper1-rq} and using \eqref{Sbx-def}, we obtain
\begin{equation} \label{122}
(q-1) S_q(1,x)
< 2 \Bigl( \log_2 \Bigl(\frac{x}{q}\Bigr)
+ C_1 + \frac{1}{\log q}\Bigr),
\end{equation}
where
$C_1 = -0.4152617906$ and $x\ge q^2$. 
Combining $\eqref{boundf5}$, \eqref{112}-\eqref{1123} and $\eqref{122}$ we have
\begin{align} \label{boundf6}\notag
S_1:=  \sum_{\chi \ne \chi_0} \sum_{p \le x_1}  \frac{\chi(p)}{p} 
  & 
  \le 
  \log_2 q + \log \ell(q)+ 2C_1 - 0.261497 + \frac{2}{\log q}
  \\& 
  < 
  \log_2 q + \log \ell(q) -1.09202 +  \frac{2}{\log q}
\end{align} 
and
\begin{align}
\label{boundf63}
  S_1
  >  
  -\log_2 q - \log \ell(q) -0.261498 - \frac{1}{\log q} .
\end{align}

We will now proceed to estimate $S_3$.
By orthogonality and the partial summation formula, we have
\begin{align}
\label{C-estim3} \notag
 S_3 = \sum_{\chi \ne \chi_0} \sum_{p \ge x_2}  \frac{\chi(p)}{p} 
  & =  
(q-1) \sum_{\substack{ p \ge x_2 \\ p \equiv 1 \pmod*{q}} } \frac{1}{p} 
- 
\sum_{\substack{p\ge x_2 \\ p \ne q}} \frac{1}{p}
\\ \notag
&
=  
\frac{1}{q} + \lim_{y \to \infty} \frac{((q-1)\pi(y;q,1)-\pi(y))}{y} - \frac{(q-1)\pi(x_2;q,1)-\pi(x_2)}{x_2} 
\\& 
+ \int_{x_2}^{\infty}  \frac{(q-1)\pi(u;q,1)-\pi(u)}{u^2} du
\ll 
q^2 e^{-c_1 \sqrt{q}},
\end{align}
where $c_1>0$ is an absolute constant. In the final estimate, 
we have used both the prime number theorem and the Siegel-Walfisz theorem.

It remains to estimate $S_2$.
Recall now (see, e.g., \cite[Ch.~19]{Davenport}) that if $\chi$ is a non-principal 
character modulo $q$ and $2 \le T \le x$, then
\begin{equation}
\label{explicit-formula}
\theta(x,\chi) 
:= \sum_{p \le x} \chi(p) \log p
= 
- \delta_{\beta_0} \frac{x^{\beta_0}}{\beta_0} 
- \sideset{}{'}\sum_{|\gamma| \le T} \frac{x^{\rho}}{\rho} 
+ O\Bigl( \frac{x (\log qx)^2 }{T} + \sqrt{x}\Bigr),
\end{equation}
where $\delta_{\beta_0} = 1$ if the Siegel zero $\beta_0$ exists and  is zero otherwise, and
$\sum^\prime$ is the sum over all non-trivial zeros 
$\rho = \beta + i \gamma$ 
of $L(s,\chi)$, with the exception of $\beta_0$ and 
its symmetric zero $1-\beta_0$.

By the partial summation formula and \eqref{explicit-formula} with $T=q^4$, we have
\begin{align}\label{late}
\notag
S_2 := \sum_{\chi \ne \chi_0} 
&\sum_{x_1 < p \le x_2} \frac{\chi(p)}{p} 
 = 
 \sum_{\chi\ne \chi_0} 
 \Bigl( 
 \frac{\theta(x_2, \chi)}{ x_2 \log x_2} 
 - \frac{\theta(x_1, \chi)}{ x_1 \log x_1} 
 + \int_{x_1}^{x_2} \theta( u,\chi) \frac{1 +\log u}{ (u \log u)^2} du 
 \Bigr)
 \\& 
 = - \delta_{\beta_0} \int_{x_1}^{x_2}  \frac{u^{\beta_0-2}}{ \log u}du 
 - \int_{x_1}^{x_2}\Bigl( \sum_{\chi \ne \chi_0} 
\sideset{}{'}\sum_{|\gamma| \le q^4} u^{\rho - 2}\Bigr) \frac{du}{\log u} 
+ (q-1) E_q,
\end{align}
and
\begin{align}
\notag 
E_q
\ll \int_{x_1}^{x_2} \Bigl(\frac{(\log qu)^{2}}{q^4 u }  +  \frac{1}{u^{3/2}} \Bigr) \frac{du}{\log u} 
\ll \frac{1}{q^2}.
\end{align} 
By using \cite[Lemmas~7 and~8]{LuZhang}\footnote{
%\label{boost-footnote}
Note that 
\cite[Lemma~7]{LuZhang} holds for every $T$ and $x_1$ such that 
$\lim_{q \to \infty} \log(qT) /\log x_1 = 0$. This  allows us to choose 
$T=q^4$ and $x_1= q^{\ell(q)}$, where $\ell(q)$ tends to infinity
arbitrarily slowly and monotonically as $q$ tends to 
infinity. The final error term in 
\cite[Lemma~8]{LuZhang} is then $\ll 1/\ell(q)=o(1)$, as $q$ tends to infinity.}, we obtain 
\begin{equation}
\label{Lemma7-LuZhang}
\int_{x_1}^{x_2}\Bigl( \sum_{\chi \ne \chi_0} 
\sideset{}{'}\sum_{|\gamma| \le q^4} u^{\rho - 2}\Bigr) \frac{du}{\log u}
\ll \frac{1}{\ell(q)}.
\end{equation}

In this case we have that  $\delta_{\beta_0} =1$ in \eqref{late};  
we now proceed to evaluate the term depending on $\beta_0$.
A direct computation using that 
$\log x_2 = q$ gives
\begin{equation*}
\int_{x_1}^{x_2}  \frac{u^{\beta_0-2}}{ \log u}du 
=  \int_{\log x_1}^{\log x_2} \frac{dt}{te^{(1-\beta_0)t}}
= E_1(1-\beta_0) 
- \int_{1-\beta_0}^{( 1-\beta_0)\log x_1} \frac{dt}{te^t}
- E_1( q(1-\beta_0)), 
\end{equation*}
where $E_1(u)$ denotes the exponential integral function.
Recalling that $x_1= q^{\ell(q)}$, 
where $\ell(q)$ tends to infinity arbitrarily slowly and monotonically as $q$ tends to 
infinity, we  have 
\begin{equation}
\label{siegel-zero-term-tails}
\int_{1-\beta_0}^{( 1-\beta_0)\log x_1} \frac{dt}{te^t}
\le 
\log_2 x_1 =
\log_2 q + \log \ell(q)
\quad
\textrm{and}
\quad
E_1(q(1-\beta_0)) 
\ll \frac{1}{q}.
\end{equation}

Inserting \eqref{Lemma7-LuZhang}-\eqref{siegel-zero-term-tails}
into  \eqref{late}, we finally get
\begin{equation}\label{f3}
\vert S_2 + E_1( 1-\beta_0) \vert
\le \log_2 q 
+ \log \ell(q)
+o(1).
\end{equation}
Combining \eqref{boundf6}-\eqref{C-estim3} and
\eqref{f3}, in this case we obtain
\begin{equation}
\label{E1-eval}
\Sigma_1 + E_1( 1-\beta_0) 
< 2\log_2 q 
+ 2\log \ell(q) %-0.722
-1.0920
\end{equation}
and
\begin{align}
\label{E1-eval3}
\Sigma_1 + E_1( 1-\beta_0)  > -2\log_2 q 
-2 \log \ell(q) - 0.2615.
\end{align}
Part \ref{Siegelzero3} of  Theorem \ref{rq-direct3} now follows on 
combining \eqref{starting3}, \eqref{Sigma2-estim3} and 
\eqref{E1-eval}-\eqref{E1-eval3} (recall that $\Sigma_2$ is bounded in \eqref{Sigma2-estim3}).
 
\subsection{Proof of Theorem \ref{rq-direct3}, Part \ref{nosiegelzero3}}
We proceed now to prove Part \ref{nosiegelzero3} of  Theorem \ref{rq-direct3}.

The quantities  $S_1,S_2,S_3$ are the same ones defined into Section \ref{proof_part2}.
We first remark that for $S_3$ we can re-use eq.~\eqref{C-estim3}.
We now estimate $S_2$.
In this case $\delta_{\beta_0} =0$ and, arguing as 
in \eqref{late}-\eqref{Lemma7-LuZhang}, we have
\begin{equation}\label{f33}
 S_2  \ll   \frac{1}{\ell(q)}.
\end{equation}
We now estimate $S_1$: 
since $\delta_{\beta_0} = 0$,
we can use a sharper version of the Brun-Titchmarsh theorem.
In particular, we can use  \eqref{BT-estim-strong-M-I} 
with $\eps=2\xi$. Since $\theta(x;q,b) = \psi(x;q,b) +O(\sqrt{x})$,
we conclude that, for $x \ge q^{7.999}$ and $b$ coprime with $q$,
\begin{equation}
\label{BT-estim-strong-M}
\theta(x;q,b)  < 2(1-\xi) \frac{x}{\varphi(q)} + C \sqrt{x},
\end{equation}
where $C>0$ is a suitable constant.
Using \eqref{BT-estim-strong-M} we can replace \eqref{122} with
\begin{equation} \label{122-strong-I-M}
(q-1) S_q(1,x)
< 2(1-\xi) \log_2 x + 2C_1 + \frac{c}{\log q},
\end{equation}
for $x>q^{8}$,
where $c>0$ is an effective constant.

A way to prove \eqref{122-strong-I-M} for $x>q^8$ is the following. 
By the partial summation formula and 
using \eqref{BT-estim-strong-M} we find
\begin{align}
\notag
(q-1) \sum_{\substack{kq < p \leq x\\ p \equiv 1 \pmod*{q}} } \frac{1}{p}
&= 
(q-1)\Bigl(\frac{\theta(x;q, 1)}{x\log x} - \frac{\theta(kq; q,1)}{kq \log(kq)} 
+ \int_{kq}^{x} \theta(u; q,1) \frac{1+ \log u}{(u \log u)^2} du\Bigr) 
\\
\notag
&
<
2(1-\xi)
\Bigl(\frac{1}{ \log x} 
+
\int_{kq}^{x} 
\Bigl(1 +\frac{1}{\log u} \Bigr) \frac{du}{u \log u} \Bigr) 
+ \frac{C}{q}
\\
\notag
&\leq  
2(1-\xi)  \bigl(\log_2 x - \log_2 (kq)\bigr) 
+
\frac{c}{\log q}
\\&
\label{large-primes-M}
\le 
2(1-\xi) \log_2 x - 2\log_2 k +\frac{c}{\log q},
\end{align}
where $c>0$ is an effective constant.
In deriving \eqref{large-primes-M} 
we also used that $x>q^8$ is equivalent to $\sqrt{x} < x^{3/4}/q^2$.
{}From eq.~(25) of \cite{paper1-rq} we also have
\begin{equation}
\label{small-primes-M}
(q-1)\sum_{\substack{p \leq kq \\ p \equiv 1 \pmod*{q}} } \frac{1}{p}  
 \le
   \sum_{j=1}^{(k-1)/2}  \frac{q-1}{2jq-1}
< \frac{1}{2} \sum_{j=1}^{(k-1)/2} \frac{1}{j}
= \frac{1}{2} H_{\frac{k-1}{2}},
\end{equation}
where $H_{n} := \sum_{j=1}^{n}\frac{1}{j}$ is the $n$-th harmonic number.
Letting $c_1(k):=\frac{1}{4} H_{\frac{k-1}{2}}  - \log_2 k$,
we now choose $k$ such that $c_1(k)$ is minimal.
It is not hard to see that  $k=55$ and that $c_1(55)  < C_1 = -0.4152617906$.
Inequality \eqref{122-strong-I-M} then follows on 
combining \eqref{large-primes-M}-\eqref{small-primes-M}.

Using \eqref{f33}, \eqref{122-strong-I-M}, 
\eqref{112}-\eqref{1123} and arguing as in \eqref{boundf6}-\eqref{boundf63}, 
we can replace \eqref{E1-eval}-\eqref{E1-eval3} with
\begin{equation}
\label{E1-evalf-strong-M}
\Sigma_1 
< 
(1-2\xi)\log_2 q + \log \ell(q) -1.0920
< 
(1-\xi)\log_2 q -1.0920,
\end{equation}
respectively
\begin{align}
\label{E1-eval3f-strong-M} 
\Sigma_1   
> 
- (1-2\xi) \log_2 q  - \log \ell(q) - 0.2615
>
- (1-\xi) \log_2 q  - 0.2615.
\end{align}

Now Part \ref{nosiegelzero3}  of  Theorem \ref{rq-direct3} follows
on combining \eqref{starting3}, \eqref{Sigma2-estim3} and \eqref{E1-evalf-strong-M}-\eqref{E1-eval3f-strong-M}.
\qed

\section{Numerical results}
\label{sec:numerical}

\newcommand{\pariversion}{PARI/GP (v.~2.15.4)}
\newcommand{\pythonversion}{Python (v.~3.11.7)}
\newcommand{\matplotlibversion}{Matplotlib (v.~3.8.2)}
\newcommand{\pandasversion}{Pandas (v.~2.2.0)}
\newcommand{\fftwversion}{FFTW (v.~3.3.10)}
\newcommand{\ubuntuversion}{Ubuntu~22.04.3~LTS} 
\newcommand{\clusteraddress}{\url{https://hpc.math.unipd.it}}
\newcommand{\capriaddress}{\url{https://capri.dei.unipd.it}}
\newcommand{\optiplexmachine}{Dell OptiPlex-3050, equipped with an Intel i5-7500 processor, 3.40GHz, 
32 GB of RAM}

The numerical results were obtained using the Fast Fourier Transform method
already presented in \cite{Languasco2021a}, see also \cite{LanguascoR2021}.

In particular, the fundamental formula for the odd Dirichlet characters case was already fully described
in \cite{paper1-rq} and reads
\begin{equation}
\label{chi-Bernoulli-method-formula2}
\sum_{\chi (-1) = -1}
\log L(1,\chi)
=  
\frac{q-1}{2} \Bigl( \log \pi - \frac{\log q}{2}  \Bigr) 
+ 
\sum_{\chi (-1) = -1}
\log \ \Bigl\vert
\sum_{a=1}^{q-1}  \frac{a}{q} \chi(a) 
\Bigl\vert.
\end{equation}  
For even $\chi$ we used the formula:
\[
L(1,\chi)
= 
2 \frac{\tau(\chi)}{q}
\sum_{a=1}^{q-1} \overline{\chi}(a)\log \Bigl(\Gamma \bigl(\frac{a}{q} \bigr)\Bigr),
\]
where $\Gamma$ denotes Euler's Gamma function and the \emph{Gau\ss\ sum} 
$\tau(\chi):= \sum_{a=1}^q \chi(a)\, e(a/q)$, $e(x):=\exp(2\pi i x)$, verifies 
$\vert \tau(\chi) \vert = q^{1/2}$ 
(see e.g., Cohen \cite[proof of Proposition 10.3.5]{Cohen2007}). 
Combination of these two formulas gives
\begin{equation}
\label{even-case}
\sum_{\substack{\chi \ne \chi_0\\ \chi (-1) = 1}}
\log L(1,\chi)
=  
\frac{q-3}{2} \Bigl(\log2 - \frac{\log q}{2}  \Bigr) 
+ 
\sum_{\substack{\chi \ne \chi_0\\ \chi (-1) = 1}}
\log \ \Bigl\vert
\sum_{a=1}^{q-1}   \overline{\chi}(a) \log \Bigl(\Gamma \bigl(\frac{a}{q} \bigr)\Bigr)
\Bigl\vert.
\end{equation}  
Hence, $\log \RR(q)$ is obtained by summing the quantities in 
\eqref{chi-Bernoulli-method-formula2}-\eqref{even-case}
and $\RR(q)$ is computed as $\exp( \log \RR(q))$. 

An alternative approach can be based on the formula
$L(1,\chi)
=
-
\frac{1}{q}
\sum_{a=1}^{q-1} \chi(a)
\digamma \bigl(\frac{a}{q}\bigr), 
$
where $\digamma(x)=(\Gamma^\prime/\Gamma)(x)$ is the \emph{digamma} function.
Very similar computations then lead to
\begin{equation}
\label{Hasse-method-formula2}
\log \RR(q)
= 
- 
(q-2) \log q 
+
\sum_{\chi \ne \chi_0}  
\log \ \Bigl\vert
\sum_{a=1}^{q-1} \chi(a)
\digamma \bigl(\frac{a}{q}\bigr)
\Bigr\vert .
\end{equation}
In practice, though, it is better to use \eqref{chi-Bernoulli-method-formula2}-\eqref{even-case} 
because in half of the cases no evaluations of special functions are needed there.
Moreover, the $\log\Gamma$ function is directly available in the C programming language.
Nevertheless, formula \eqref{Hasse-method-formula2} can be useful to double-check
our results.

The summation over $a$ in  \eqref{chi-Bernoulli-method-formula2}-\eqref{even-case} 
can be handled using the FFT procedure and we can also
embed here a \emph{decimation in frequency strategy}, see, e.g., 
\cite{paper1-rq,Languasco2021a,LanguascoR2021}.
The FFT procedure requires $O(q)$ memory positions and
the computation of $\log \RR(q)$ via 
\eqref{chi-Bernoulli-method-formula2}-\eqref{even-case}  
has a computational cost of $O(q\log q)$ 
arithmetic operations plus the cost of computing $q-1$ values of
the $\log\Gamma$ and logarithm functions and products. 

For the evaluation of the computational error we refer to Section 6.3 of \cite{paper1-rq}.
More statistics and details on computations regarding $L(1,\chi)$
can be found in \cite{Languasco2021,Languasco2023}.

\subsection{Comments on the plots and on the histograms} 
The actual values of $\RR(q)$ 
presented in the herewith included plots and histograms
were obtained for every odd prime $q$ up to $\bound$
using the FFTW \cite{FrigoJ2005} software library
set to work with the \textit{long double precision} (80 bits). 
Such results were then collected in some \emph{comma-separated values} (csv) files and then
all the plots and the histograms were
obtained running on such stored data some suitable designed scripts written
using \pythonversion\ and making use of the packages \pandasversion\ and \matplotlibversion.

Figure \ref{fig1} shows the values for $\RR(q)$ for $q$ up to $\bound$;
it is clear that $\RR(q)$ essentially behaves as 
%$c/\log q$, 
$c/(\log q)^{3/4}$, where $1/5<c<2/3$. This is compatible with the estimates in Theorem \ref{rq-direct3}.

In Figure \ref{fig2} we present their normalized values $\RR(q) (\log q)^{3/4}$.
Figures \ref{fig3} shows the histograms of the same quantities.

\subsection{Computing the prime sums over Dirichlet characters}
Here we briefly show how to compute $\log \RR(q)$ 
using prime sums over Dirichlet characters. 
This procedure is much less efficient 
than the one that uses the Fast Fourier Transform. 
However, for small values of $q$ it can be used to double-check the results.
Moreover, it also shows a way to independently compute 
$\Sigma_1$ and $\Sigma_2$.

Recalling \eqref{starting3},
we first split the sum over primes (this is important
to improve the convergence speed, see Lemma \ref{lz-lemma} below) so that
\begin{align*}
\sum_p \log\Bigl(1-\frac{\chi(p)}{p}\Bigr)
 &=
  \sum_{p\le P} \log\Bigl(1-\frac{\chi(p)}{p}\Bigr)
+
 \sum_{p>P} \sum_{m \ge 1} \frac{\chi(p^m)}{mp^m},
  \\&
 =
  \sum_{p\le P} \log\Bigl(1-\frac{\chi(p)}{p}\Bigr)
+
 \sum_{m \ge 1} \frac{1}{m} \sum_{p>P}  \frac{\chi^m(p)}{p^m},
\end{align*}
where $P=Aq$, $A$ is a fixed positive integer\footnote{In practical computations 
$A$, and hence $P$ too, cannot be not too large since 
to perform this step the whole set of prime numbers up to $P$ must be generated.} 
and $\chi$ is a non-principal character.
Since the principal character is not involved, there are no 
convergence problems. We also used the multiplicativity of the Dirichlet characters.

Hence
\begin{equation}
%\notag
\log \RR(q)
%&
=  
\sum_{\chi \ne \chi_0} \sum_{p\le P} \log\Bigl(1-\frac{\chi(p)}{p}\Bigr)
+
\sum_{\chi \ne \chi_0}  \sum_{m \ge 1} \frac{1}{m} \sum_{p>P}  \frac{\chi^m(p)}{p^m}
\label{first-split}
=: 
r(q,P) + S(q,P),
\end{equation}
say.
Let $M\ge 2$ be an integer.
Splitting the sum over $m$, we obtain 
\begin{align}
\notag
\label{E1-def}
S(q,P)
&= 
\sum_{\chi \ne \chi_0} \sum_{m =1}^{M} \frac{1}{m}  \sum_{p>P} \frac{\chi^m(p)}{p^m}
+
\sum_{\chi \ne \chi_0} \sum_{m > M} \frac{1}{m} \sum_{p>P}  \frac{\chi^m(p)}{p^m}
\\&
=:
S(q,P,M) + E_{1}(q,P,M),
\end{align}
say.
By  trivial bounds on $\chi$, it is  easy to prove that
\begin{equation}
\label{E1-estim}
  \vert E_{1}(q,P,M)  \vert
  \leq
  \frac{P (q-1)}{M (M - 1) (P - 1) P^{M}}.
\end{equation}
Using the M\"obius inversion formula, see, e.g., Cohen \cite[Proposition 10.1.5]{Cohen2007},
we can write
\[
  \sum_{p>P} \frac{\chi^m(p)}{p^m}
  = 
  \sum_{k \ge 1}
    \frac{\mu(k)}k \log\bigl(L_{P}(km,\chi^{km}) \bigr),
\]
where, for $\Re(s)\ge1$ and $\chi\ne \chi_0$, we have defined
the truncated $L$-function as
\begin{equation*}
L_{P}(s,\chi) 
  :=
  \prod_{p > P} \Bigl( 1 - \frac{\chi(p)}{p^s} \Bigr)^{-1}
  =
  L(s,\chi) \prod_{p \le P} \Bigl( 1 - \frac{\chi(p)}{p^s} \Bigr)
  .
\end{equation*}
Hence
\begin{equation}
\label{SqPM-def}
S(q,P,M)
=
\sum_{\chi \ne \chi_0} 
\sum_{m =1}^{M} 
\frac{1}{m}
  \sum_{k \ge 1}
    \frac{\mu(k)}k \log\bigl(L_{P}(km,\chi^{km}) \bigr).
\end{equation}
Remark that $km=1$ only if $k=m=1$. Recalling  $\chi\ne\chi_0$, in the previous formula
we will never encounter the pole at $1$ of the Riemann zeta function.

Let now $K\ge 1$ be an integer. Splitting the sum over $k$ in \eqref{SqPM-def}, we have 
\begin{align}
\notag
S(q,P,M)
&=
\sum_{\chi \ne \chi_0} 
\sum_{m =1}^{M} 
\frac{1}{m}
\sum_{k = 1}^{K}
    \frac{\mu(k)}k \log\bigl(L_{P}(km,\chi^{km}) \bigr)
\\&
\notag
\hskip1cm+
\sum_{\chi \ne \chi_0} 
\sum_{m =1}^{M} 
\frac{1}{m}
\sum_{k > K}
    \frac{\mu(k)}k \log\bigl(L_{P}(km,\chi^{km}) \bigr)
\\&
\label{second-split}
= :
S(q,P,M,K)  + E_{2}(q,P,M,K),
\end{align}
say.
Combining \eqref{first-split}-\eqref{E1-def} and \eqref{second-split} we obtain
\begin{equation}
\label{rq-comput}
\log \RR(q) =  
r(q,P) + S(q,P,M,K)  + E_{1}(q,P,M) + E_{2}(q,P,M,K).
\end{equation}

Both $r(q,P)$ and $S(q,P,M,K)$ are finite sums 
and can be directly computed, while $E_{1}(q,P,M)$ and $E_{2}(q,P,M,K)$
must be estimated.
For $E_{1}(q,P,M)$ we will use \eqref{E1-estim}. 
To estimate $E_{2}(q,P,M,K)$, we will need the following 
lemma \cite[Lemma 1]{lanzac09}. 
\begin{lem}
\label{lz-lemma}
Let $\chi \pmod*{q}$ be a Dirichlet character and $n \ge 2$ be an integer.
If $P \ge 1$ is an integer then
\[
  \bigl| \log \bigl( L_{P}(n,\chi) \bigr) \bigr|
  \le
  \frac{P^{1 - n}}{n - 1}.
\]
\end{lem}
\noindent
Using this lemma it is easy to see that
\begin{equation}
\label{E2-estim}
  \vert E_{2}(q,P,M,K) \vert
  \leq
  \frac{2 P (q-1)}{K^2 (P - 1) (P^{K}-1)},
\end{equation}
where the estimate does not depend on $M$.

Let now $\Delta>0$ be an integer.
Taking $P=Aq$, by exploiting \eqref{E1-estim} and \eqref{E2-estim} one can choose $M, K$
such that $\vert E_{1}(q,P,M)  \vert + \vert E_{2}(q,P,M,K) \vert <10^{-\Delta}$.
Taking $P=Aq$, by exploiting \eqref{E1-estim} and \eqref{E2-estim} one can choose $M, K$
such that $\vert E_{1}(q,P,M)  \vert + \vert E_{2}(q,P,M,K) \vert <10^{-\Delta}$,
where $\Delta>0$ is any prescribed integer.
Hence, by computing $r(q,P) + S(q,P,M,K)$ in \eqref{rq-comput}, 
one obtains $\log \RR(q)$ with an accuracy 
of (at least) $\Delta$ decimals.

\smallskip
We remark that the same idea can be used for $\Sigma_2$, since it is enough to 
let $m$ start from $2$ in the analogue of $S(q,P)$.
For $\Sigma_1$, the algorithm is simpler since there is no sum over $m$ in \eqref{E1-def}, 
in $S(q,B,1)$ and in $S(q,B,1,K)$. Hence only $E_{2}(q,B,1,K)$ is present.
Moreover, if one has already computed $\log \RR(q)$ with sufficient
accuracy using the FFT, then it is enough to obtain $\Sigma_1$ 
in order to have $\Sigma_2=  \log \RR(q) - \Sigma_1$ too.

\section{Connections and analogies with the Mertens' constants \texorpdfstring{\\}{} 
in arithmetic progressions}
\label{sec:mertens-analogies}
In this final section we establish some connections and analogies for the prime sum
that defines $ \log \RR(q)$  with the ones involved in the definition of the
Mertens and Meissel-Mertens constants in arithmetic progressions:
$\MM(q,b), \BB(q,b), \CC(q,b)$, $1\le b < q$, $(b,q)=1$, since they can be 
written with sums over Dirichlet characters
that are similar to the one into \eqref{starting3}. 

Recalling that $q$ is prime and
using eq.~(1-1) in \cite{lanzac10b}, we have
\begin{equation*}
\sum\limits_{\substack{p\leq x\\ p\equiv b \pmod*{q}}}
\frac{1}{p}
=
\frac{\log \log x}{q-1} + \MM(q,b)+ O_{q}\Bigl(\frac{1}{\log x}\Bigr),
\end{equation*}
where $x \to + \infty$.
Moreover,  a direct computation
and Theorem 428 of Hardy-Wright \cite{HardyW2008} show that
\begin{equation*}
  (q-1) \MM(q, b)
  =
  \MM - \sum_{p\mid q}  \frac{1}{p}
  +
  \sum_{\chi \ne \chi_0} \overline{\chi}(b)\sum_p \frac{\chi(p)}{p},
\end{equation*}
where $\MM$ is the Meissel-Mertens constant defined in \eqref{Meissel-Mertens-def}.
From the previous equation, and using \eqref{starting3}, it is clear that 
\[
\Sigma_1 =(q-1) \MM(q, 1) - \MM +  \sum_{p\mid q} 
 \frac{1}{p}.
\]

Moreover, ${\mathcal R}(q)$ is related with the constant for the Mertens'
product in arithmetic progression. In \cite[p.~38]{lanzac07} it is proved
that 
\begin{equation*}
\prod_{\substack{p \le x\\ { p \equiv b \pmod*{q}}}} 
\Bigl(1 - \frac{1}{p}\Bigr)
\sim 
\frac{\CC(q,b)}{(\log x)^{\frac{1}{q-1}}},
\end{equation*}
as $q$ tends to infinity, where $\CC(q,b)$ 
verifies
\[
  \CC(q, b)^{q-1}
  =
  e^{-\Euler}
  \prod_p
    \Bigl( 1 - \frac 1p \Bigr)^{\alpha(p; q, b)},
\]
and $\alpha(p; q, b) = q - 2$ if $p \equiv b \bmod q$ and
$\alpha(p; q, b) = - 1$ otherwise.
In  \cite[p.~316]{lanzac09} (or \cite[eq.~(2-4)]{lanzac10b}) it is 
also proved that 
\begin{equation*}
  (q-1)
  \log \CC(q,b)
  =
  -\Euler
  + \log \frac{q}{q-1}
  -
  \sum_{\chi \ne \chi_0}
    \overline{\chi}(b)
    \sum_{m \ge 1} \frac1m \sum_p \frac{\chi(p)}{p^m}.
\end{equation*}
Besides some correction terms, 
the only difference between the formula for $\log \CC(q,1)$
and the one in \eqref{starting3} for 
$\log \RR(q)$ 
is that $\chi(p^m)$ in \eqref{starting3} must be replaced by $\chi(p)$.

Finally, for  $\BB(q,b)$ defined as
\[
\BB(q,b)
:=
\sum_{p\equiv b \bmod{q}} 
\Bigl(\log\bigl(1-\frac{1}{p}\bigr)+\frac{1}{p} \Bigr),
\]
in \cite[eq.~(2-3)]{lanzac10b} it is proved that
\begin{equation*}
  (q-1) \BB(q,b)
 :=
    \BB(q)
  -
  \sum_{\chi \ne \chi_0}
    \overline{\chi}(b)
    \sum_{m \ge 2}\frac1m \sum_p \frac{\chi(p)}{p^m},
\end{equation*}
where $\BB(q)  := - \sum_{m \ge 2}\frac1m \sum_{(p,q)=1} \frac{1}{p^m},$
represents the contribution of the principal character
$\chi_0 \pmod*{q}$ 
and equals, cf.\,\eqref{Meissel-Mertens-def},
\[
  \BB(q) 
  =
  \sum_{(p,q)=1} \Bigl(\log\bigl(1-\frac{1}{p}\bigr)+\frac{1}{p} \Bigr)
  =
  {\mathcal M}-\Euler - \sum_{p\mid q}  \Bigl(\log\bigl(1-\frac{1}{p}\bigr)+\frac{1}{p} \Bigr).
\]  
In this case too, besides some correction terms, 
the only difference between the formula for $\BB(q,1)$ 
and the one in \eqref{starting3} for  $\Sigma_2$  is that $\chi(p^m)$ in 
\eqref{starting3} must be replaced by $\chi(p)$.

The behavior of the prime sum over characters that involves $\chi^m(p)-\chi(p)$
is studied in \cite{lanzac10}, starting from eq.~(6)-(7) there.

\medskip
\noindent \textbf{Acknowledgment}. 
The authors would like to thank the referee for his/her suggestions
and thorough review of this work. In particular, (s)he provided a version of Lemma \ref{lem1-referee}, which we present here in improved form.
Work on this paper began in 2023 during a postdoctoral stay of the first author at 
the Max-Planck-Institut f\"ur Mathematik (Bonn) and completed
at the beginning of 2024 when she started
a postdoc at the University of Hong Kong. Both she and the third author thank MPIM 
and its staff for the hospitality and excellent working conditions. 
The first author extends gratitude to the third author for hosting 
and providing mentorship during her stay at MPIM.
The  computational work was carried out on machines
of the cluster located at the Dipartimento di Matematica ``Tullio Levi-Civita'' of 
the University of Padova, see \url{https://hpc.math.unipd.it}. 
The authors are grateful for having had such computing facilities 
at their disposal. 

%\newpage

%%COMMENTED FIGURES AND HISTOGRAMS
\ifthenelse{\boolean{plots_included}}
{
\begin{figure}[H]
\includegraphics[scale=0.35,angle=0]{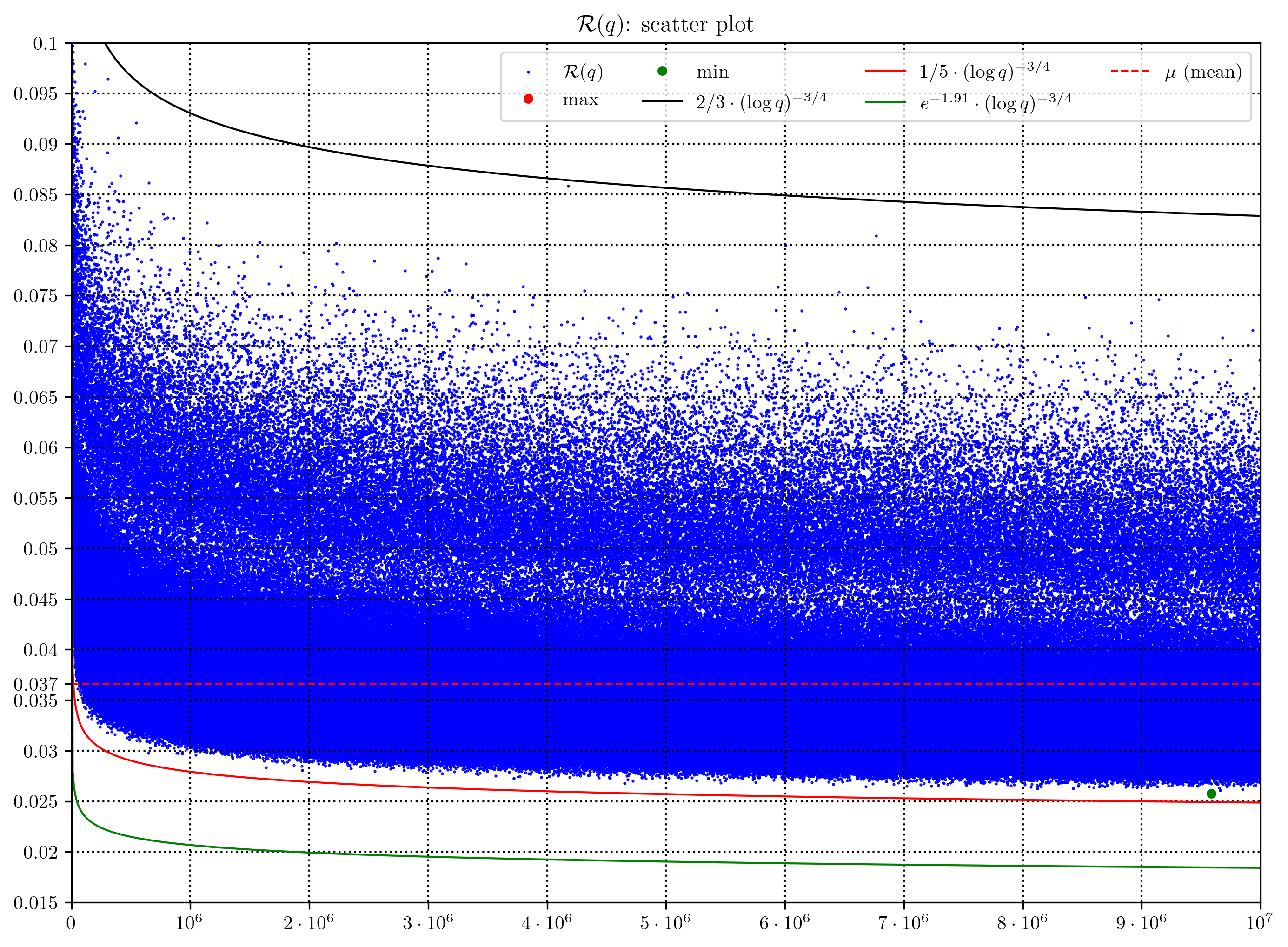} 
\caption{{\small The values of $\RR(q)$, $q$ prime, $3\le q\le  \bound$.
The maximal value (red dot) is attained at $q=3$ and its value is $0.604599\dots$;
much larger than the other plotted values.
The red dashed line represents the mean value.
}}
\label{fig1}
\end{figure}  

\begin{figure}[H]
\includegraphics[scale=0.35,angle=0]{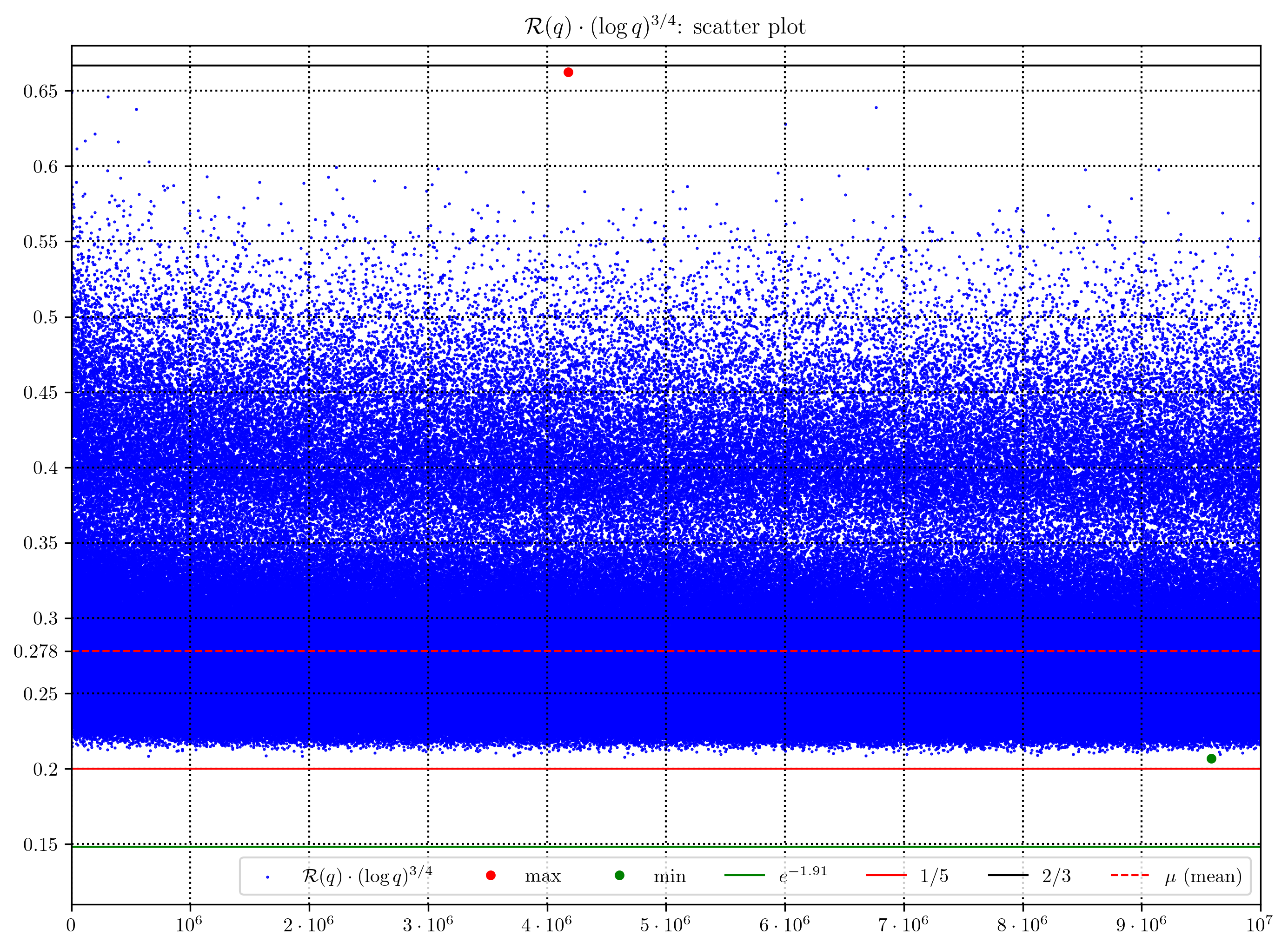} 
\caption{{\small The values of $\RR(q) ( \log q)^{3/4}$, $q$ prime, $3\le q\le  \bound$.
The red dashed line represents the mean value.
}}
\label{fig2}
\end{figure}  

\begin{figure}[H]
\scalebox{0.85}{
\begin{minipage}{0.5\textwidth} 
\includegraphics[scale=0.35,angle=0]{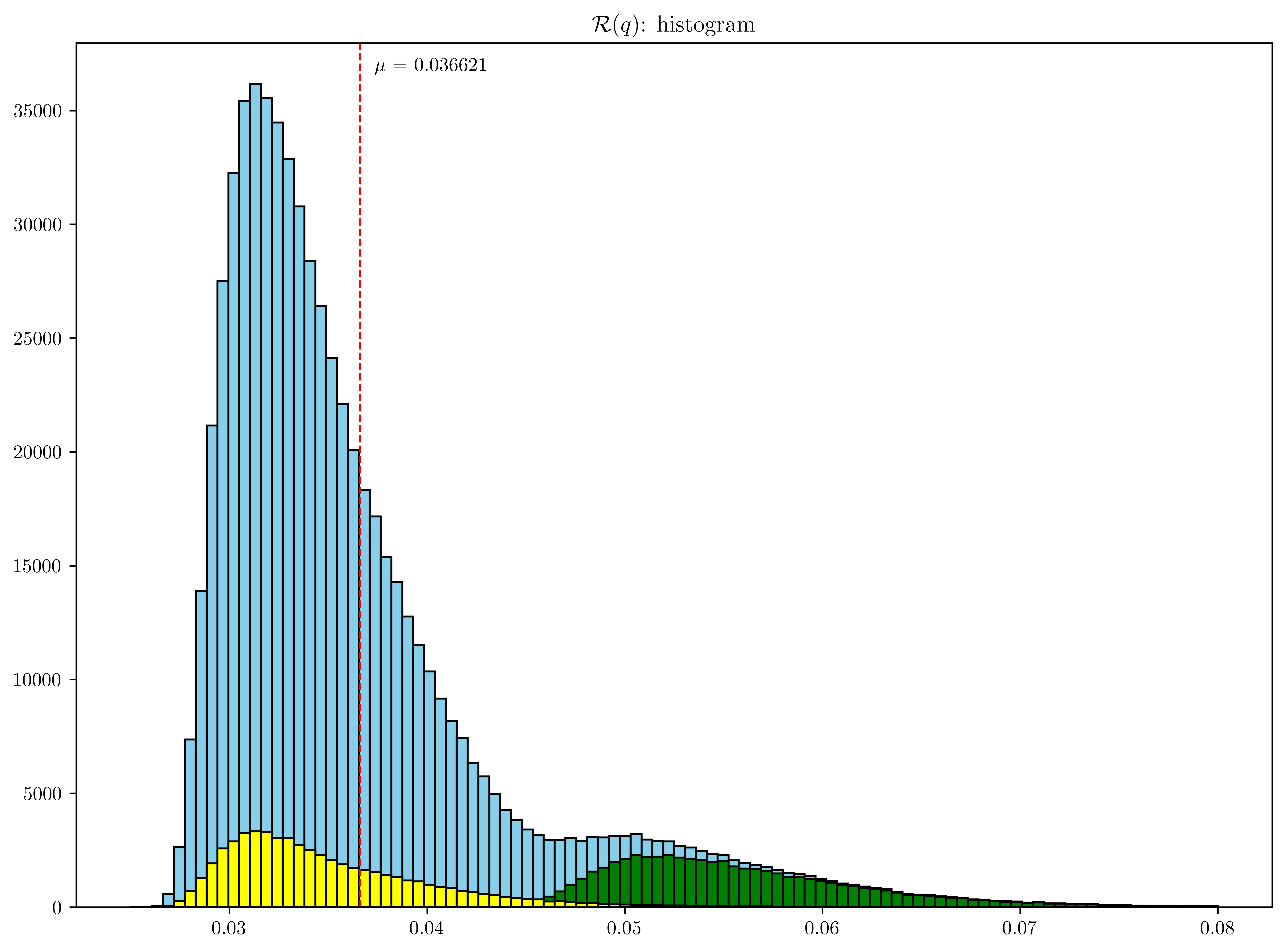}  
\end{minipage} 
\begin{minipage}{0.5\textwidth} 
\includegraphics[scale=0.35,angle=0]{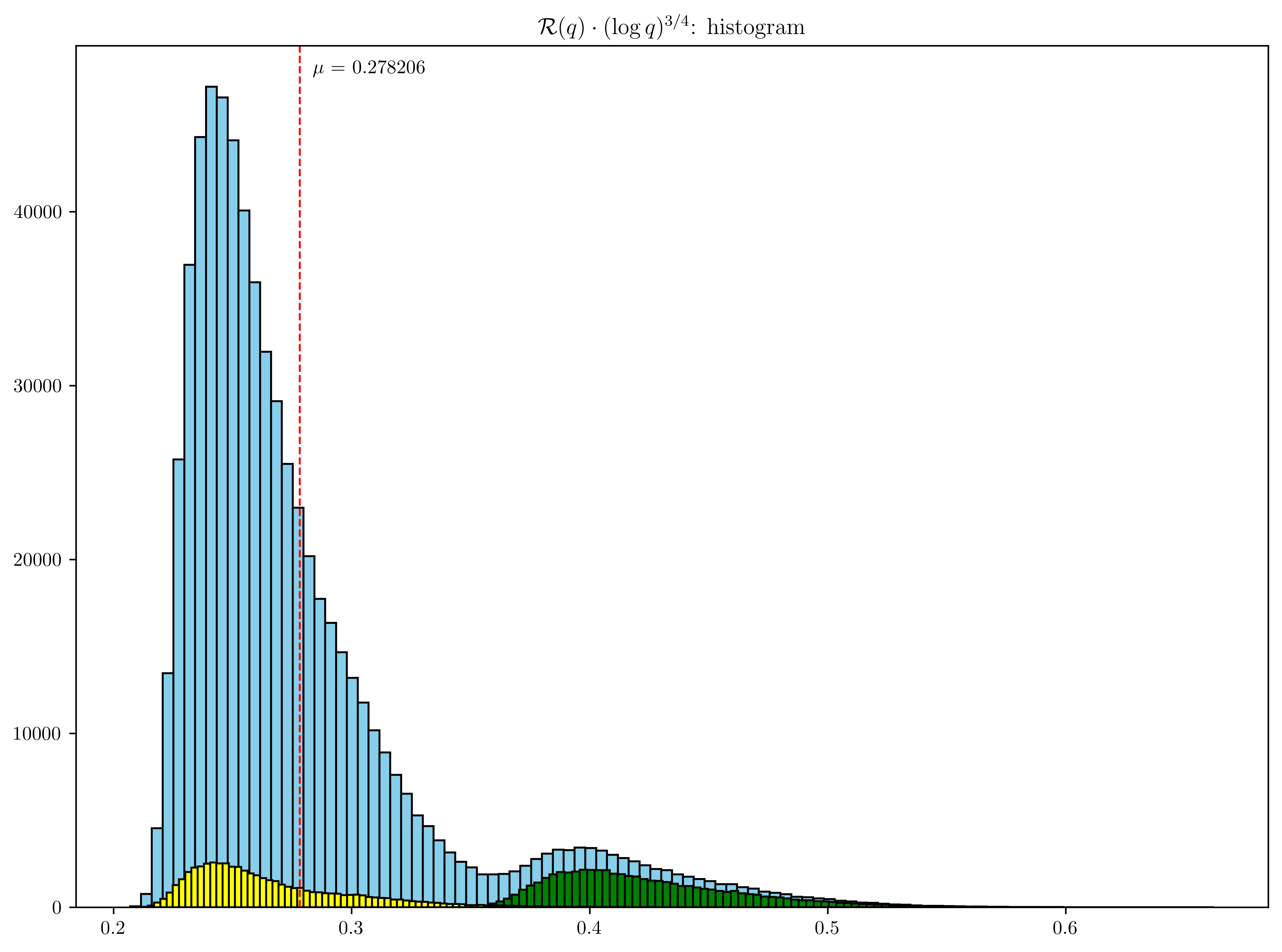}  
\end{minipage} 
}
\caption{On the left: the values of $\RR(q)$ (cerulean bars), $q$ prime, $3\le q\le  \bound$, 
but the contributions of the primes $q \ge 5$ such that $2q+1$ is prime (green bars) or $2q-1$ is prime 
(yellow bars) are superimposed.
On the right: idem, but for the normalized values $\RR(q) (\log q)^{3/4}$.
The red dashed lines represent the mean values.
}
\label{fig3}
\end{figure}  
}

%\bibliographystyle{abbrv}
%\bibliographystyle{alpha}

%\bibliography{biblio}

\begin{thebibliography}{10}

\bibitem{AC}
N.~C. Ankeny and S.~Chowla.
\newblock The class number of the cyclotomic field.
\newblock {\em Proc. Nat. Acad. Sci. U.S.A.}, 35:529--532, 1949.

\bibitem{Bessassi}
S.~Bessassi. 
\newblock Bounds for the degrees of CM-fields of class number one.
\newblock {\em Acta Arith.} 106 (2003), 213--245.

\bibitem{Boyadzhiev2007}
K. N.~Boyadzhiev.
\newblock Evaluation of series with Hurwitz and Lerch zeta function coefficients by using Hankel contour integrals,
\newblock {\em Appl. Math. Comput.},
186 (2007), 1559--1571.

\bibitem{chokim}
P.~J. Cho and H.~H. Kim.
\newblock Extreme residues of {D}edekind zeta functions.
\newblock {\em Math. Proc. Cambridge Philos. Soc.}, 163(2):369--380, 2017.

\bibitem{Cohen2007}
H.~Cohen.
\newblock {\em {Number} {Theory}. {Volume II}: {Analytic} and {Modern}
  {Tools}}, volume 240 of {\em Graduate Texts in Mathematics}.
\newblock Springer, 2007.

\bibitem{Davenport}
H.~Davenport.
\newblock {\em {Multiplicative Number Theory}}, volume~74 of {\em Graduate
  Texts in Mathematics}.
\newblock Springer-Verlag, New York, third edition, 2000.
\newblock Revised and with a preface by Hugh L. Montgomery.

\bibitem{dusart}
P.~Dusart.
\newblock Explicit estimates of some functions over primes.
\newblock {\em Ramanujan J.}, 45(1):227--251, 2018.

\bibitem{FriedlanderI1997}
J.~Friedlander and H.~Iwaniec.
\newblock The {B}run-{T}itchmarsh theorem.
\newblock In {\em Analytic {N}umber {T}heory ({K}yoto, 1996)}, volume 247 of
  {\em London Math. Soc. Lecture Note Ser.}, pages 85--93. Cambridge Univ.
  Press, Cambridge, 1997.

\bibitem{FrigoJ2005}
M.~Frigo and S.~Johnson.
\newblock {T}he {D}esign and {I}mplementation of {FFTW}3.
\newblock {\em Proceedings of the IEEE}, 93:216--231, 2005.
\newblock Available from \url{https://www.fftw.org}.

%\bibitem{Gr}
%A.~Granville.
%\newblock On the size of the first factor of the class number of a cyclotomic
%  field.
%\newblock {\em Invent. Math.}, 100(2):321--338, 1990.

\bibitem{HardyW2008}
G.~H. Hardy and E.~M. Wright.
\newblock {\em An {I}ntroduction to the {T}heory of {N}umbers}.
\newblock Oxford University Press, Oxford, sixth edition, 2008.
\newblock Revised by D. R. Heath-Brown and J. H. Silverman, With a foreword by
  Andrew Wiles.

\bibitem{paper1-rq}
N.~Kandhil, A.~Languasco, P.~Moree, S.~S. Eddin, and A.~Sedunova.
\newblock The {K}ummer ratio of the relative class number for prime cyclotomic
  fields.
%\newblock {\em submitted}, \href{https://arxiv.org/abs/2402.13829}{arXiv:2402.13829}, 2024.
\newblock {\em J. Math. Anal. Appl.} 538 (2024), no. 1, Paper No. 128368, 24 pp.

%\bibitem{Lang1994}
%S.~Lang.
%\newblock {\em Algebraic {N}umber {T}heory}, volume 110 of {\em Graduate Texts
%  in Mathematics}.
%\newblock Springer-Verlag, New York, second edition, 1994.

\bibitem{Languasco2021a}
A.~Languasco.
\newblock Efficient computation of the {E}uler-{K}ronecker constants of prime
  cyclotomic fields.
\newblock {\em Res. Number Theory}, 7(1):Paper No. 2, 22, 2021.

\bibitem{Languasco2021}
A.~Languasco.
\newblock Numerical verification of {L}ittlewood's bounds for {$|L(1,\chi)|$}.
\newblock {\em J. Number Theory}, 223:12--34, 2021.

\bibitem{Languasco2023}
A.~Languasco.
\newblock Numerical estimates on the {L}andau-{S}iegel zero and other related
  quantities.
\newblock {\em J. Number Theory}, 251:185--209, 2023.

\bibitem{LanguascoR2021}
A.~Languasco and L.~Righi.
\newblock A fast algorithm to compute the {R}amanujan-{D}eninger gamma function
  and some number-theoretic applications.
\newblock {\em Math. Comp.}, 90(332):2899--2921, 2021.

\bibitem{lanzac07}
A.~Languasco and A.~Zaccagnini.
\newblock A note on {M}ertens' formula for arithmetic progressions.
\newblock {\em J. Number Theory}, 127(1):37--46, 2007.

\bibitem{lanzac09}
A.~Languasco and A.~Zaccagnini.
\newblock On the constant in the {M}ertens product for arithmetic progressions.
  {II}. {N}umerical values.
\newblock {\em Math. Comp.}, 78(265):315--326, 2009.

\bibitem{lanzac10b}
A.~Languasco and A.~Zaccagnini.
\newblock Computing the {M}ertens and {M}eissel-{M}ertens constants for sums
  over arithmetic progressions.
\newblock {\em Experiment. Math.}, 19(3):279--284, 2010.
\newblock With an appendix by Karl K. Norton.

\bibitem{lanzac10}
A.~Languasco and A.~Zaccagnini.
\newblock On the constant in the {M}ertens product for arithmetic progressions.
  {I}. {I}dentities.
\newblock {\em Funct. Approx. Comment. Math.}, 42:17--27, 2010.

\bibitem{Li12}
X.~Li.
\newblock The smallest prime that does not split completely in a number field.
\newblock {\em Algebra Number Theory}, 6(6):1061--1096, 2012.



\bibitem{Lou15}
S.~R. Louboutin.
\newblock Explicit upper bounds for residues of {D}edekind zeta functions.
\newblock {\em Mosc. Math. J.}, 15(4):727--740, 2015.

\bibitem{Lou15b}
S.~R. Louboutin.
\newblock Real zeros of Dedekind zeta functions. 
\newblock {\em Int. J. Number Theory} 11 (2015), 843--848.

\bibitem{LuZhang}
Y.~Lu and W.~Zhang.
\newblock On the {K}ummer conjecture.
\newblock {\em Acta Arith.}, 131(1):87--102, 2008.

\bibitem{Maynard2013}
J.~Maynard.
\newblock On the {Brun-Titchmarsh} theorem.
\newblock {\em Acta Arith.}, 157(3):249--296, 2013.

\bibitem{MVsieve}
H.~L. Montgomery and R.~C. Vaughan.
\newblock The large sieve.
\newblock {\em Mathematika}, 20:119--134, 1973.

\bibitem{Motohashi1979}
Y.~Motohashi.
\newblock A note on {S}iegel's zeros.
\newblock {\em Proc. Japan Acad. Ser. A Math. Sci.}, 55(5):190--191, 1979.

\bibitem{Ramare2009}
O.~Ramar\'{e}.
\newblock {\em Arithmetical {A}spects of the {L}arge {S}ieve {I}nequality},
  volume~1 of {\em Harish-Chandra Research Institute Lecture Notes}.
\newblock Hindustan Book Agency, New Delhi, 2009.
%\newblock With the collaboration of D.S.~Ramana.

%\bibitem{Ribenboim2001}
%P.~Ribenboim.
%\newblock {\em Classical {T}heory of {A}lgebraic {N}umbers}.
%\newblock Universitext. Springer-Verlag, New York, 2001.

\bibitem{stark}
H.~M. Stark.
\newblock Some effective cases of the {B}rauer-{S}iegel theorem.
\newblock {\em Invent. Math.}, 23:135--152, 1974.

\bibitem{Tatu}
T.~Tatuzawa.
\newblock On the product of {$L(1,\chi)$}.
\newblock {\em Nagoya Math. J.}, 5:105--111, 1953.

\bibitem{Ulmer}
D.~Ulmer.
\newblock On the {B}rauer-{S}iegel ratio for abelian varieties over function
  fields.
\newblock {\em Algebra Number Theory}, 13(5):1069--1120, 2019.

\end{thebibliography}

\bigskip\noindent Neelam Kandhil  \par\noindent
{\footnotesize 
University of Hong Kong, Department of Mathematics, Pokfulam, Hong Kong.
\hfil\break
e-mail: {\tt kandhil@hku.hk}}

\medskip\noindent Alessandro Languasco \par\noindent
{\footnotesize Universit\`a di Padova,
Department of Information Engineering - DEI \hfil\break
via Gradenigo 6/b, 35131 Padova, Italy.\hfil\break
e-mail: {\tt alessandro.languasco@unipd.it}}

\medskip\noindent Pieter Moree  \par\noindent
{\footnotesize Max-Planck-Institut f\"ur Mathematik,
Vivatsgasse 7, D-53111 Bonn, Germany.\hfil\break
e-mail: {\tt moree@mpim-bonn.mpg.de}}

\end{document}